\documentclass[11pt]{article}
\usepackage{epigamath}

\usepackage[notext]{kpfonts}
\usepackage{baskervald}

\setpapertype{A4}

\usepackage[english]{babel}

\usepackage[dvips]{graphicx}     

\title{Limits of the trivial bundle on a curve}
\titlemark{Limits of the trivial bundle on a curve}
\author{\vspace{0cm} Arnaud Beauville}
\authoraddresses{
\authordata{Arnaud Beauville}{\firstname{Arnaud} \lastname{Beauville}\\
\institution{Universit\'e C\^ote d'Azur,
CNRS -- Laboratoire J.-A. Dieudonn\'e,
Parc Valrose,
\hbox{F-06108} Nice cedex 2, France}\\
\email{arnaud.beauville@unice.fr}}
}
\authormark{A. Beauville}
\date{\vspace{-5ex}} 
\journal{\'Epijournal de G\'eom\'etrie Alg\'ebrique} 
\acceptation{Received by the Editors on March 8, 2018, and in final form
on August 30, 2018. \\ Accepted on October 4, 2018.}

\newcommand{\articlereference}{Volume 2 (2018), Article Nr. 9}

 \usepackage[all]{xy}

\allowdisplaybreaks

\newtheorem{lem}{Lemma}
\newtheorem{prop}{Proposition}

\newtheorem{rem}{Remark}

\def\rond{\kern 1pt{\scriptstyle\circ}\kern 1pt}

\newcommand\Ker{\operatorname{Ker}}

\newcommand\Pic{\operatorname{Pic}}

\newcommand{\im}{\operatorname{Im} }

\newcommand\abs[1]{\lvert {#1}\rvert}

\renewcommand{\thesubsection}{(\thesection.\arabic{subsection})}

\def\C{\mathbb{ C}}
\def\P{\mathbb{ P}}

\def\A{\mathbb{ A}}

\def\O{\mathcal{ O}}

\def\iso{\vbox{\hbox to .8cm{\hfill{$\scriptstyle\sim$}\hfill}
\nointerlineskip\hbox to .8cm{{\hfill$\longrightarrow $\hfill}} }}

\catcode`\@=11\def\eqalign#1{\null\,\vcenter{\openup\jot\m@th\ialign{
\strut\hfil$\displaystyle{##}$&$\displaystyle{{}##}$
&&\quad\strut$\displaystyle{##}$&$\displaystyle{{}##}$
\crcr#1\crcr}}\,}
\catcode`\@=12

\begin{document}


\maketitle



\begin{prelims}


\def\abstractname{Abstract}
\abstract{We attempt to describe the vector bundles on a curve $C$ which are specializations of $\mathcal{O}_C^2\,$. We get a complete classification when $C$ is Brill-Noether-Petri general, or when it is hyperelliptic; in both cases all  limit vector bundles are decomposable. We give examples of indecomposable  limit bundles for some special curves.}



\vspace{0.15cm}

\languagesection{Fran\c{c}ais}{%

\textbf{Titre. Limites du fibr\'e trivial sur une courbe} \commentskip \textbf{R\'esum\'e.} Nous essayons de d\'ecrire les fibr\'es vectoriels qui sont des sp\'ecialisations de $\mathcal{O}_C^{2}$. Nous obtenons une classification compl\`ete lorsque $C$ est g\'en\'erale au sens de Brill-Noether-Petri, ou lorsque $C$ est hyperelliptique; les fibr\'es limites sont d\'ecomposables dans chacune des deux situations. Nous donnons \'egalement des exemples de fibr\'es limites ind\'ecomposables sur certaines courbe sp\'eciales.}

\end{prelims}


\newpage

\setcounter{tocdepth}{1} \tableofcontents

\section{Introduction}
Let $ C $ be a smooth  complex projective curve, and $ E $ a vector bundle on $ C $, of rank $ r $. We will say that $ E $ \emph{is a limit of} $\O_C^r$ if there exists an algebraic family $ (E_b)_{b\in B}  $ of vector bundles on $ C $, parametrized by an algebraic curve $ B $, and a point $ \mathrm{o}\in B $, such that $ E_\mathrm{o}=E $ and $ E_b\cong \O_C^r $ for $ b\neq \mathrm{o} $. Can we classify all these vector bundles? If $E$ is a limit of $\O_C^2$ clearly $E\oplus \O_C^{r-2}$ is a limit of $\O_C^r$, so it seems reasonable to start in rank 2.

We get a complete classification in two extreme cases: when $C$ is generic (in the sense of Brill-Noether theory), and when it is hyperelliptic. In both cases the limit vector bundles are of the form $L\oplus L^{-1}$, with some precise conditions on $L$. However for large families of curves, for instance for plane curves, some limits of $\O_C^2$ are indecomposable, and those seem hard to classify.

\medskip	
\section{Generic curves}
Throughout the paper we denote by $C$  a smooth connected  projective curve of genus $g$ over $\C$.

\begin{prop}\label{ex}
Let $L$ be a line bundle on $C$ which is a limit of globally generated line bundles \emph{(}in particular, any line bundle of degree $\geq g+1)$. Then $L\oplus L^{-1}$ is a limit of $\O_C^2\,$. 
\end{prop}

\begin{proof}
By hypothesis there exist a curve $B$, a point $\mathrm{o}\in B$ and a line bundle $\mathcal{L}$ on $C\times B$ such that $\mathcal{L}_{|C\times \{\mathrm{o}\} }\cong L$ and $\mathcal{L}_{|C\times \{\mathrm{b}\} }$ is globally generated for $b\neq \mathrm{o}$. We may assume that $B$ is affine and that $\mathrm{o}$ is defined by $f=0$ for a global function $f$ on $B$; we put $B^*:=B\smallsetminus\{\mathrm{o}\} $.

We choose two general sections $s,t$ of $\mathcal{L}$ on $C\times B^*$; reducing $B^*$ if necessary, we may assume that they generate $\mathcal{L}$. Thus we have an exact sequence on $C\times B^*$
\[0\rightarrow \mathcal{L}^{-1}\xrightarrow{\ (t,-s)\ } \O_{C\times B^*}^2\xrightarrow{\ (s,t)\ } \mathcal{L}\rightarrow 0\]which corresponds to an extension class $e\in H^1(C\times B^*,\mathcal{L}^{-2})$. For $n$ large enough, $f^ne$ comes from a class in $H^1(C\times B,\mathcal{L}^{-2})$ which vanishes along $C\times \{\mathrm{o}\} $; this class gives rise to an extension
\[0\rightarrow \mathcal{L}^{-1}\longrightarrow  \mathcal{E}\longrightarrow\mathcal{L}\rightarrow 0\]with $\mathcal{E}_{|C\times \{b\} }\cong \O_C^2$ for $b\neq \mathrm{o}$, and $\mathcal{E}_{|C\times \{\mathrm{o}\} }\cong L\oplus L^{-1}$.\qed
\end{proof}

\begin{rem}\label{prop}{\rm
Let $E$ be a vector bundle limit of $\O_C^2\,$. We have $\det E=\O_C\,$, and $h^0(E)\geq 2$ by semi-continuity. If  $E$ is semi-stable this implies $E\cong\O_C^2\,$; otherwise  $E$ is unstable. Let $L$ be the maximal destabilizing sub-line bundle of $E$; we have an extension $0\rightarrow L\rightarrow E\rightarrow L^{-1}\rightarrow 0$, with $h^0(L)\geq 2$. Note that this extension is trivial (so that $E=L\oplus L^{-1}$) if $H^1(L^2)=0$, in particular if $\deg(L)\geq g$.} \end{rem}

\begin{prop}
Assume that $ C $ is Brill-Noether-Petri general. The following conditions are equivalent:
\begin{itemize}
\item[\rm (i)] $ E $ is a limit of $ \O_C^2 \,;$
\item[\rm (ii)] $ h^0(E)\geq 2  $ and $ \det E=\O_C \, ;$
\item[\rm (iii)] $ E=L\oplus L^{-1} $ for some line bundle $ L $ on $ C $ with $ h^0(L)\geq 2  $ or $ L=\O_C \,.$
\end{itemize}
 \end{prop}
\begin{proof}  We have seen that (i)   implies  (ii)   (Remark \ref{prop}). Assume  (ii)   holds, with $E\not\cong \O_C^2\,$. Then $E$ is unstable, and
 we have an extension $0\rightarrow L\rightarrow E\rightarrow L^{-1}\rightarrow 0$ with $h^0(L)\geq 2$. 
Since $ C $ is Brill-Noether-Petri general we have $ H^0(C,K_C\otimes L^{-2})=0  $ \cite[Ch.\ 21, Proposition 6.7]{ACG}, hence $ H^1(C,L^2)=0  $. Therefore the above extension is trivial, and we get  (iii).
 
 \smallskip
 Assume that (iii)   holds. 
  Brill-Noether theory implies that 
   any line bundle $ L $ with $ h^0(L)\geq 2  $ is a limit of globally generated ones 
 \footnote{Indeed, the subvariety $W^r_d$ of $\Pic^d(C)$  parametrizing  line bundles $L$ with $h^0(L)\geq  r+1$ is equidimensional, of  dimension $g-(r+1)(r+g-d)$; the line bundles which are not globally generated belong to the subvariety $W^r_{d-1}+C$, which has codimension~$r$.}.
   So (i) follows from Proposition \ref{ex}.\qed
 \end{proof}

 \section{Hyperelliptic curves}
 
 \begin{prop}
Assume that $C$ is hyperelliptic, and let $H$ be the line bundle on $C$ with $h^0(H)=\deg(H)=2$. The limits of $\O_C^2$ are the decomposable bundles $L\oplus L^{-1}$, with $\deg(L)\geq g+1$ or $L=H^k$ for $k\geq 0$.
\end{prop}

\begin{proof} Let $\pi :C\rightarrow \P^1$ be the two-sheeted covering defined by $\lvert H\rvert$. Let us say that an effective divisor $D$ on $C$ is \emph{simple} if it does not contain a divisor of the form $\pi ^*p$ for $p\in\P^1$.
We will need the following well-known lemma:

\begin{lem}
\label{lemma1}
Let $L$ be a line bundle on $C$.
\begin{itemize}
\item[\rm 1)] If $L=H^k(D)$ with  $D$ simple and $\deg(D)+k\leq g$, we have $h^0(L)=h^0(H^k)=k+1$.
\item[\rm 2)] If $\deg(L)\leq g$,  $L$ can be written in a unique way $H^k(D)$ with $D$ simple. If $L$  is globally generated, it is a power of $H$.
\end{itemize}
\end{lem}

\begin{proof}[Proof of Lemma \ref{lemma1}] 1) Put $\ell:=g-1-k$ and $d:=\deg(D)$. Recall that $K_C\cong H^{g-1}$. Thus by Riemann-Roch, the first assertion is equivalent to $h^0(H^{\ell }(-D))=h^0(H^{\ell })-d$. 
We have $H^0(C,H^{\ell })=\allowbreak \pi ^*H^0(\P^1,\O_{\P^1}(\ell))$; since $D$ is simple of degree $\leq \ell+1$, it imposes $d$
 independent conditions on $H^0(C,H^{\ell })$, hence our claim.

\smallskip	
2) Let $k$ be the greatest integer such that $h^0(L\otimes H^{-k})>0$; then $L=H^k(D)$ for some effective divisor $D$, which is simple since $k$ is maximal.  By 1) $D$ is the fixed part of $\abs{L}$, hence is uniquely determined, and so is $k$.
In particular the only globally generated line bundles on $C$ of degree $\leq g$ are the powers of $H$.
\qed
\end{proof}
 
 \noindent\emph{Proof of the Proposition} : Let $E$ be a vector bundle on $C$ limit of $\O_C^2\,$. Consider the exact sequence 
 \begin{equation}\label{hyp}
0\rightarrow L\rightarrow E\rightarrow L^{-1}\rightarrow 0\, ,
\end{equation}
 where
 we can assume $\deg(L)\leq g$ (Remark \ref{prop}). By  Lemma \ref{lemma1} we have $L=H^k(D)$ with $D$ simple of degree $\leq g-2k$. After tensor product with $H^k$, the corresponding cohomology exact sequence reads 
 \[0\rightarrow H^0(C,H^{2k}(D))\rightarrow H^{0}(C,E\otimes H^k)\rightarrow H^0(C,\O_C(-D))\xrightarrow{\ \partial \ } H^1(C,H^{2k}(D))\]
 which implies 
 $h^0(E\otimes H^k)= h^0(H^{2k}(D))+\dim\Ker \partial =2k+1+\dim\Ker \partial \ $ by Lemma~\ref{lemma1}. 
 
 By semi-continuity we have $h^0(E\otimes H^k)\geq 2h^0(H^k)=2k+2$; the only possibility is  $D=0$ and $\partial =0$.
 But $\partial (1)$ is the class of the extension (\ref{hyp}), which must therefore be trivial; hence $E=H^k\oplus H^{-k}$.\qed
 \end{proof}

\section{Examples of indecomposable limits}

To prove that some limits of $\O_C^2$ are indecomposable we will need the following easy lemma:

\begin{lem}\label{split}
Let $L$ be a line bundle of positive degree on $C$, and let \begin{equation}\label{2}
0\rightarrow L\rightarrow E\rightarrow L^{-1}\rightarrow 0
\end{equation}  be an exact sequence. The following conditions are equivalent:
\begin{itemize}
\item[\rm (i)] $E$ is indecomposable;
\item[\rm (ii)] The extension \emph{(\ref{2})} is nontrivial;
\item[\rm (iii)] $h^0(E\otimes L)=h^0(L^2)$.
\end{itemize}
\end{lem}

\begin{proof} The implication (i)$\ \Rightarrow\ $(ii) is clear.

(ii)$\ \Rightarrow\ $(iii) : After tensor product with $L$, the  cohomology exact sequence associated to (\ref{2}) gives
\[0\rightarrow H^0(L^2)\xrightarrow{\ i\ } H^0(E\otimes L)\longrightarrow H^0(\O_C)\xrightarrow{\ \partial \ }H^1(L^2)\, , \]where $\partial $ maps $1\in H^0(\O_C)$ to the extension class of (\ref{2}). Thus (ii) implies that  $i$ is an isomorphism, hence (iii).

(iii)$\ \Rightarrow\ $(i): If $E$ is decomposable, it must be equal to $L\oplus L^{-1}$ by unicity of the  destabilizing  bundle. But this  implies $h^0(E\otimes L)=h^0(L^2)+1$. \qed
\end{proof}

The following construction was suggested by N. Mohan Kumar:

\begin{prop}
\label{proposition4}
Let $C\subset \P^2$ be a smooth plane curve, of degree $d$. For $0<k<\dfrac{d}{4} $, there exist extensions
\[0\rightarrow \O_{C}(k)\rightarrow E \rightarrow \O_C(-k)\rightarrow 0\] such that $E$ is indecomposable and is a limit of $\O_C^2\,$.
\end{prop}

\begin{proof} 
Let $Z$ be a finite subset of $\P^2$ which is the complete intersection of two curves of degree $k$, and such that $C\cap Z=\varnothing$. By \cite[Remark 4.6]{S}, for a general extension
\begin{equation}
\label{3}
0\rightarrow \O_{\P^2}(k)\rightarrow E\rightarrow \mathcal{I}_Z(-k)\rightarrow 0\, ,
\end{equation}
the vector bundle $E$ is a limit of $\O_{\P^2}^2$; therefore $E_{|C}$ is a limit of $\O_C^2\,$.

The extension (\ref{3}) restricts to an exact sequence
\[0\rightarrow \O_C(k)\rightarrow E_{|C}\rightarrow \O_C(-k)\rightarrow 0\, .\]

To prove that $E_{|C}$ is indecomposable, it suffices by Lemma \ref{split}  to prove that 
$ h^0(E_{|C}(k))=h^0(\O_C(2k))$. Since $2k<d$ we have $h^0(\O_C(2k))=h^0(\O_{\P^2}(2k))=h^0(E(k))$, so in view of the exact sequence
\[0\rightarrow E(k-d)\longrightarrow E(k)\longrightarrow E_{|C}(k)\rightarrow 0\]it suffices to prove  $H^1(E(k-d))=0$, or  by Serre duality $H^1(E(d-k-3))=0$.

The exact sequence (\ref{3}) gives an injective map $H^1(E(d-k-3))\hookrightarrow H^1(\mathcal{I}_Z(d-2k-3))$. Now since $Z$ is a complete intersection we have an exact sequence
\[0\rightarrow \O_{\P^2}(-2k)\rightarrow \O_{\P^2}(-k)^2\rightarrow \mathcal{I}_Z\rightarrow 0\, ;\]
since $4k<d$ we have $H^2(\O_{\P^2}(d-4k-3))=0$, hence $H^1(\mathcal{I}_Z(d-2k-3))=0$, and finally $H^1(E(d-k-3))=0$ as asserted.\qed
\end{proof}

\medskip	
We can also perform the Str\o mme construction directly on the curve $C$, as follows.
Let $L$ be a base point free line bundle on $C$. We choose sections $s,t\in H^0(L)$ with no common zero. This gives rise to a Koszul extension
\begin{equation}\label{K}
0\rightarrow L^{-1}\xrightarrow{\ i\ } \O_C^2 \xrightarrow{\ p\ } L\rightarrow 0\quad\mbox{with }\  i=(-t,s)\,,\ p=(s,t)\,.\end{equation}
We fix a nonzero  section $u\in H^0(L^2)$. Let $\mathcal{L}$ be the pull-back of $L$ on $C\times \A^1$.  We consider  the complex (``monad")
\[\mathcal{L}^{-1}\xrightarrow{\ \alpha \ }\mathcal{L}^{-1}\oplus \O^2\oplus \mathcal{L} \xrightarrow{\ \beta \ }\mathcal{L}\ ,\qquad \alpha =(\lambda  ,i,u)\,,\ \beta =(u,p,-\lambda ),  \]where $\lambda $ is the coordinate on $\A^1$.
 Let $\mathcal{E}:=\Ker \beta /\im \alpha$, and let $E:=\mathcal{E}_{|C\times \{0\} }$.
 
 \begin{lem}
$E$ is a rank 2 vector bundle, limit of $\O_C^2\,$. There is  an exact sequence \allowbreak$0\rightarrow L\rightarrow E\rightarrow L^{-1}\rightarrow 0$; the corresponding extension class in  $H^1(L^2)$ is the product  by $u^2\in H^0(L^4)$ of the  class $e\in H^1(L^{-2})$ of the Koszul extension $(\ref{K})$.
\end{lem}

\begin{proof} The proof is essentially the same as in \cite{S}; we give the details for completeness. 
 
 For $\lambda  \neq 0$, we get easily $\mathcal{E}_{|C\times \{\lambda \}  }\cong\O_C^2\,$; we will show that $E$ is a rank 2 vector bundle. This implies that $\mathcal{E}$ is a vector bundle on $C\times \A^1$,
  and therefore that $E$ is a  limit of $\O_C^2\,$.
  
  Let us denote by $\alpha _0,\beta _0$ the restrictions of $\alpha $ and $\beta $ to $C\times \{0\} $.
 We have $\Ker \beta_0 =\allowbreak L\oplus N$, where $N$ is the kernel of $(u,p):L^{-1}\oplus \O_C^2\rightarrow L$. Applying the snake lemma to the commutative diagram
 \[\xymatrix@M=5pt{0\ar[r] &L^{-1}\ar[r]^{i}\ar[d]&\O^2\ar[r]^{p}\ar@{^{(}->}[d]&L \ar[r]\ar@{=}[d]&0\\
0\ar[r] &N\ar[r]& L^{-1}\oplus \O^2 \ar[r]& L \ar[r]&0\\
}\]we get an exact sequence 
\begin{equation}\label{N}
0\rightarrow L^{-1}\rightarrow N\rightarrow L^{-1}\rightarrow 0\,,
\end{equation} which fits into a commutative diagram
\[\xymatrix{0\ar[r] &L^{-1}\ar[r]\ar@{=}[d]&N\ar[r]\ar[d]&L^{-1} \ar[r]\ar[d]^{\times u}&0\,\hphantom{.}\\
0\ar[r] &L^{-1}\ar[r]&  \O^2 \ar[r]& L \ar[r]&0\, ;
}\]this means that the extension (\ref{N}) is the pull-back by $\times u:L^{-1}\rightarrow L$ of the Koszul extension (\ref{K}).

Now since $E$ is the cokernel of the map $L^{-1}\rightarrow L\oplus N$ induced by $\alpha _0$, we have a commutative diagram 
\[\xymatrix{0\ar[r] &L^{-1}\ar[r]\ar[d]^{\times u}&N\ar[r]\ar[d]&L^{-1} \ar[r]\ar@{=}[d]&0\\
0\ar[r] &L\ar[r]&  E \ar[r]& L^{-1} \ar[r]&0
}\]so that the extension $L\rightarrow E\rightarrow L^{-1}$ is the push-forward by $\times u$ of  (\ref{N}). This implies the Lemma.\qed
\end{proof}

Unfortunately it seems difficult in general to decide whether the extension $L\rightarrow E\rightarrow L^{-1}$ nontrivial. Here is a case where we can conclude:

\begin{prop}\label{theta}
Assume that $C$ is non-hyperelliptic. Let $L$ be a globally generated line bundle on $C$ such that $L^2\cong K_C$. Let $0\rightarrow L\rightarrow E\rightarrow L^{-1}\rightarrow 0$ be the unique nontrivial extension of $L^{-1}$ by $L$. Then $E$ is indecomposable, and is a limit of $\O_C^2\,$.
\end{prop}

\begin{proof} We choose   $s,t$ in $H^0(L)$ without common zero, and use the previous construction.  It suffices to prove that we can choose $u\in H^0(K_C)$ so that $u^2e\neq 0$: since $H^1(K_C)\cong\C$, the vector bundle $E$ will be the unique  nontrivial extension of $L^{-1}$ by $L$, and indecomposable by Lemma \ref{split}.
 
Suppose that $u^2e=0$ for all $u$ in $H^0(K_C)$; by bilinearity this implies $uve=0$ for all $u,v$ in $H^0(K_C)$. Since $C$ is not hyperelliptic, the multiplication map
$\mathsf{S}^2H^0(K_C)\rightarrow H^0(K_C^2)$ is surjective, so we have $we=0$ for all $w\in H^0(K^2)$. But the pairing
\[H^1(K_C^{-1})\otimes H^0(K_C^2)\rightarrow H^1(K_C)\cong\C\]is perfect by Serre duality, hence our hypothesis implies $e=0$, a contradiction. \qed
\end{proof}

\begin{rem}{\rm
In the moduli space $\mathcal{M}_g$ of curves of genus $g\geq 3$, the curves $C$ admitting a line bundle $L$ with $L^2\cong K_C$ and $h^0(L)$ even $\geq 2$ form an irreducible divisor \cite{T2}; for a general curve $C$ in this divisor, the line bundle $L$ is unique, globally generated, and satisfies $h^0(L)=2$ \cite{T1}.
Thus  Proposition \ref{theta} provides for $g\geq 4$ a  codimension 1 family of curves in $\mathcal{M}_g$ admitting an indecomposable vector bundle limit of $\O_C^2\,$.}
\end{rem}

 \begin{rem}{\rm
Let $\pi :C\rightarrow B$ be a finite morphism of smooth projective curves. If $E$ is a vector bundle limit of $\O_B^2\,$, then clearly $\pi ^*E$ is  a limit of $\O_C^2\,$. Now if $E$ is  indecomposable, $\pi ^*E$ is also  indecomposable. Consider indeed the nontrivial extension $0\rightarrow L\rightarrow E\rightarrow L^{-1}\rightarrow 0$ (Remark \ref{prop});  by Lemma \ref{split} it suffices to show that the class  $e\in H^1(B,L^2)$ of this extension remains nonzero in $H^1(C,\pi ^*L^2)$. But the
pull-back homomorphism $\pi ^*:H^1(B,L^2) \rightarrow H^1(C,\pi ^*L^2)$ can be identified with the homomorphism $H^1(B,L^2) \rightarrow H^1(B,\pi _*\pi ^*L^2)$ deduced from the linear map $L^2\rightarrow \pi _*\pi ^*L^2$, and the latter is an isomorphism onto a direct factor; hence $\pi ^*$ is injective and $\pi ^*e\neq 0$, so $E$ is indecomposable.

\smallskip	
Thus any curve dominating one of the curves considered in Propositions \ref{proposition4} and \ref{theta}   carries an indecomposable vector bundle which is a limit of $\O_C^2\,$.}
\end{rem}

\providecommand{\bysame}{\leavevmode\hbox to3em{\hrulefill}\thinspace}
%
%

\bibliographystyle{amsalpha}
\bibliographymark{References}
\def\cprime{$'$}

\end{document}